\documentclass[12pt]{amsart}

\usepackage{amsmath}
\usepackage{amssymb}
\usepackage{amsfonts}
\usepackage{amsthm}
\usepackage{enumerate}
\usepackage{hyperref}
\usepackage{color}

\theoremstyle{plain}
\newtheorem{thm}{Theorem}[section]

\newtheorem{prop}[thm]{Proposition}
\newtheorem{lem}[thm]{Lemma}
\newtheorem{cor}[thm]{Corollary}

\newtheorem{ques}[thm]{Question}

\theoremstyle{definition}

\newtheorem{dfns-rems}[thm]{Definitions and Remarks}
\newtheorem{notas-rems}[thm]{Notations and Remarks}
\newtheorem{exmps-rems}[thm]{Examples and Remarks}

\begin{document}


\title[depth and sdepth of integral closure of powers]{On the depth and Stanley depth of integral closure of powers of monomial ideals}


\author[S. A. Seyed Fakhari]{S. A. Seyed Fakhari}

\address{S. A. Seyed Fakhari, School of Mathematics, Statistics and Computer Science,
College of Science, University of Tehran, Tehran, Iran.}

\email{aminfakhari@ut.ac.ir}


\begin{abstract}
Let $\mathbb{K}$ be a field and $S=\mathbb{K}[x_1,\dots,x_n]$ be the
polynomial ring in $n$ variables over $\mathbb{K}$. Assume that $G$ is a
graph with edge ideal $I(G)$. We prove that the modules $S/\overline{I(G)^k}$ and $\overline{I(G)^k}/\overline{I(G)^{k+1}}$ satisfy Stanley's inequality for every integer $k\gg 0$. If $G$ is a non-bipartite graph, we show that the ideals $\overline{I(G)^k}$ satisfy Stanley's inequality for all $k\gg 0$. For every connected bipartite graph $G$ (with at least one edge), we prove that ${\rm sdepth}(I(G)^k)\geq 2$, for any positive integer $k\leq {\rm girth}(G)/2+1$. This result partially answers a question asked in \cite{s3}. For any proper monomial ideal $I$ of $S$, it is shown that the sequence $\{{\rm depth}(\overline{I^k}/\overline{I^{k+1}})\}_{k=0}^{\infty}$ is convergent and $\lim_{k\rightarrow\infty}{\rm depth}(\overline{I^k}/\overline{I^{k+1}})=n-\ell(I)$, where $\ell(I)$ denotes the analytic spread of $I$. Furthermore, it is proved that for any monomial ideal $I$, there exists an integer $s$ such that $${\rm depth} (S/I^{sm}) \leq {\rm depth} (S/\overline{I}),$$for every integer $m\geq 1$. We also determine a value $s$ for which the above inequality holds. If $I$ is an integrally closed ideal, we show that ${\rm depth}(S/I^m)\leq {\rm depth}(S/I)$, for every integer $m\geq 1$. As a consequence, we obtain that for any integrally closed monomial ideal $I$ and any integer $m\geq 1$, we have ${\rm Ass}(S/I)\subseteq {\rm Ass}(S/I^m)$.
\end{abstract}


\subjclass[2000]{13C15, 05E40, 13B22}


\keywords{Depth, Edge ideal, Integral closure, Stanley depth, Stanley's inequality}


\thanks{}


\maketitle


\section{Introduction} \label{sec1}

Let $\mathbb{K}$ be a field and let $S=\mathbb{K}[x_1,\dots,x_n]$
be the polynomial ring in $n$ variables over $\mathbb{K}$. Let
$M$ be a finitely generated $\mathbb{Z}^n$-graded $S$-module. Let
$u\in M$ be a homogeneous element and $Z\subseteq
\{x_1,\dots,x_n\}$. The $\mathbb {K}$-subspace $u\mathbb{K}[Z]$
generated by all elements $uv$, with $v$ a monomial in $\mathbb{K}[Z]$, is
called a {\it Stanley space} of dimension $|Z|$, if it is a free
$\mathbb{K}[Z]$-module. Here, as usual, $|Z|$ denotes the number
of elements of $Z$. A decomposition $\mathcal{D}$ of $M$ as a
finite direct sum of Stanley spaces is called a {\it Stanley
decomposition} of $M$. The minimum dimension of a Stanley space
in $\mathcal{D}$ is called the {\it Stanley depth} of
$\mathcal{D}$ and is denoted by ${\rm sdepth} (\mathcal {D})$.
The quantity $${\rm sdepth}(M):=\max\big\{{\rm sdepth}
(\mathcal{D})\mid \mathcal{D}\ {\rm is\ a\ Stanley\
decomposition\ of}\ M\big\}$$ is called the {\it Stanley depth}
of $M$. As a convention, we set ${\rm sdepth}(M)=\infty$, when $M$ is the zero module. We say that a $\mathbb{Z}^n$-graded $S$-module $M$ satisfies {\it Stanley's inequality} if $${\rm depth}(M) \leq
{\rm sdepth}(M).$$ In fact, Stanley \cite{s} conjectured that every $\mathbb{Z}^n$-graded $S$-module satisfies Stanley's inequality.
For a reader friendly introduction to Stanley depth, we refer to
\cite{psty} and for a nice survey on this topic, we refer to
\cite{h}.

The Stanley's conjecture has been recently disproved in \cite{abcj}. The counterexample presented in \cite{abcj} lives in the category of squarefree monomial ideals. Thus, one can still ask whether Stanley's inequality holds for non-squarefree monomial ideals. Based on this observation, in \cite[Question 1.1]{s5}, we asked wether the high powers of any monomial ideal satisfy Stanley's inequality. More explicit, we proposed the following question.

\begin{ques} \label{q1}
{\rm (}{\rm \cite[Question 1.1]{s5})} Let $I$ be a monomial ideal. Is it true that $I^k$ and $S/I^k$ satisfy Stanley's inequality for every integer $k\gg 0$?
\end{ques}

This question was investigated for edge ideals in \cite{asy}, \cite{psy} and \cite{s3} (see Section \ref{sec2} for the definition of edge ideals). The most general results are obtained in \cite{s3}. In that paper, we proved that if $G$ is a graph with $n$ vertices and $I(G)$ is its edge ideal, then $S/I(G)^k$ satisfies Stanley's inequality for every integer $k\geq n-1$ \cite[Corollary 2.5]{s3}. If moreover $G$ is a non-bipartite graph, or
at least one of the connected components of $G$ is a tree with at least one edge, then $I(G)^k$ satisfies Stanley's inequality for every integer $k\geq n-1$ \cite[Corollary 3.6]{s3}. Also, in \cite{s6}, we showed that Question \ref{q1} has the positive answer when $I$ is the cover ideal of a bipartite graph.

In this paper, we ask whether the answer of Question \ref{q1} is positive if one replaces $I^k$ by its integral closure. In other words, we pose the following question.

\begin{ques} \label{q2}
Let $I$ be a monomial ideal. Is it true that $\overline{I^k}$ and $S/\overline{I^k}$ satisfy Stanley's inequality for every integer $k\gg 0$?
\end{ques}

Hoa and Trung \cite[Lemma 1.5]{ht}, prove that for every monomial ideal $I$, we have $\lim_{k\rightarrow\infty}depth (S/\overline{I^k})=n-\ell(I)$, where $\ell(I)$ denotes the analytic spread of $I$. Thus, Question \ref{q2} is equivalent to the following question.

\begin{ques} \label{q3}
Let $I$ be a monomial ideal. Is it true that ${\rm sdepth}(\overline{I^k})\geq n-\ell(I)+1$ and ${\rm sdepth}(S/\overline{I^k})\geq n-\ell(I)$, for every integer $k\gg 0$?
\end{ques}

In Section \ref{sec3}, we study this question for edge ideals of graphs. Note that for any graph $G$, we have $\ell(I(G))=n-p$, where $n$ is the number of vertices and $p$ is the number of bipartite connected components of $G$ (see e.g. \cite[Page 50]{v}). Before stating our results, we mention that Stanley depth of integral closure of powers of monomial ideals was studied in \cite{s1}. One of the main results of that paper asserts that if $I_2\subseteq I_1$ are two monomial ideals, then there exists an integer $s\geq 1$, such that for every $m\geq 1$,$${\rm sdepth} (I_1^{sm}/I_2^{sm}) \leq {\rm sdepth} (\overline{I_1}/\overline{I_2})$$(see Lemma \ref{collect}). This inequality has a crucial role in this paper. As a consequence of this inequality, we will show in Theorem \ref{quo} that for any edge ideal $I=I(G)$, the module $S/\overline{I^k}$ satisfies Stanley's inequality, for $k\gg 0$. We also, prove that if $G$ is a non-bipartite graph, then $\overline{I(G)^k}$ satisfies Stanley's inequality for every integer $k\gg 0$ (see Theorem \ref{ide}).

Assume that $G$ is a bipartite graph. By \cite[Theorem 1.4.6 and Corollay 10.3.17]{hh'}, we know that $I(G)$ is a normal ideal. Thus, $\overline{I(G)^k}$ satisfies Stanley's inequality if and only if $I(G)^k$ satisfies the Stanley's inequality. We do not know whether for a bipartite graph $G$, the ideal $I(G)^k$ satisfies the Stanley's inequality, for any integer $k\gg 0$. However, in \cite{s3}, we noticed that it is sufficient to consider connected bipartite graphs. Indeed, we proved that $I(G)^k$ satisfies the Stanley's inequality, for every bipartite graph $G$ and for any integer $k \gg 0$, provided that the answer of the following question is positive.

\begin{ques} \label{q4}
{\rm (}{\rm \cite[Question 3.3]{s3})} Let $G$ be a connected bipartite graph (with at least one edge) and suppose $k\geq 1$ is an integer. Is it true that ${\rm sdepth}(I(G)^k)\geq 2$?
\end{ques}

In \cite[Proposition 3.4]{s3}, we showed that the answer of Question \ref{q4} is positive when $G$ is a tree. In Theorem \ref{girth}, we extend this result, by proving that for any connected bipartite graph $G$ and every integer $k\leq {\rm girth}(G)/2+1$, we have ${\rm sdepth}(I(G)^k)\geq 2$. Assume that $G$ is a (not necessarily connected) bipartite graph with at leat one edge and let $g$ be the maximum girth of the connected components of $G$. As a consequence of Theorem \ref{girth}, we conclude that for every integer $k\leq g/2+1$,$${\rm sdepth}(I(G)^k)\geq n-\ell(I(G))=n-p,$$ where $n$ is the number of vertices and $p$ is the number of connected components of $G$ (see Corollary \ref{discon}).

After studying the Stanley depth of $\overline{I^k}$ and $S/\overline{I^k}$, we consider the modules of the form $\overline{I^k}/\overline{I^{k+1}}$. In order to determine whether $\overline{I^k}/\overline{I^{k+1}}$ satisfies Stanley's inequality for $k\gg 0$, we need to know the asymptotic behavior of depth of these modules. We know from \cite[Theorem 1.2]{hh''} that for every proper monomial ideal $I$ of $S$, the sequence $\{{\rm depth}(I^k/I^{k+1})\}$ is convergent and $$\lim_{k\rightarrow\infty}{\rm depth}(S/I^k)=\lim_{k\rightarrow\infty}{\rm depth}(I^k/I^{k+1}).$$In Theorem \ref{dtwoquo}, we show that the same holds if one replaces the powers of $I$ by their integral closure. In other words, the sequence $\{{\rm depth}(\overline{I^k}/\overline{I^{k+1}})\}$ is convergent and moreover, $$\lim_{k\rightarrow\infty}{\rm depth}(S/\overline{I^k})=\lim_{k\rightarrow\infty}{\rm depth}(\overline{I^k}/\overline{I^{k+1}}).$$ As mentioned above, by the result of Hoa and Trung \cite[Lemma 1.5]{ht}, we know that $\lim_{k\rightarrow\infty}{\rm depth}(S/\overline{I^k})=n-\ell(I)$. Thus, in order to prove that $\overline{I^k}/\overline{I^{k+1}}$ satisfies Stanley's inequality for every integer $k\gg 0$, we must show ${\rm sdepth}(\overline{I^k}/\overline{I^{k+1}})\geq n-\ell(I)$. We prove this for edge ideals in Theorem \ref{twoquo}. We mention that the proof of Theorem \ref{twoquo} is also based on Lemma \ref{collect}.

As we mentioned above, Lemma \ref{collect} has a crucial role in Section \ref{sec3}. As a particular case of this lemma, for every monomial ideal $I\subseteq S$, there exists an integer $s\geq 1$ with the property that$${\rm sdepth}(S/I^{sm})\leq {\rm sdepth}(S/\overline{I}).$$It is reasonable to ask whether this inequality is true, if one replaces sdepth by depth. In Theorem \ref{dnormal1}, we give a positive answer to this question and even more, we show that one can choose $s$ to be $\mu(\overline{I^{\ell(I)-1}})!$, where for every monomial ideal $J$, we denote by $\mu(J)$ the number of minimal monomial generators of $J$. The proof of Theorem \ref{dnormal1} is based on a formula due to Takayama \cite[Theorem 2.2]{t1} which is a generalization of the so-called Hochster's formula and relates the local cohomology modules of a (non-squarefree) monomial ideal to reduced homologies of particular simplicial complexes.

Finally , assume that $I$ is a squarefree monomial ideal. We know from the proof of \cite[Theorem 2.6]{htt} that$${\rm depth}(S/I^m)\leq {\rm depth}(S/I),$$for every integer $m\geq 1$. In Theorem \ref{dnormal2}, we extend this inequality to the class integrally closed monomial ideals. As a consequence, we obtain that for any integrally closed monomial ideal and every integer $m\geq 1$, we have$${\rm Ass}(S/I)\subseteq {\rm Ass}(S/I^m)$$(see Corollary \ref{ass}).


\section{Preliminaries} \label{sec2}

In this section, we provide the definitions and basic facts which will be used in the next sections.

\subsection{Notions from commutative algebra}

Let $\mathbb{K}$ be a field and $S=\mathbb{K}[x_1,x_2,\dots,x_n]$ be the
polynomial ring in $n$ variables over $\mathbb{K}$. Assume that
 $I\subset S$ is an arbitrary ideal. An element $f \in S$ is
{\it integral} over $I$, if there exists an equation
$$f^k + c_1f^{k-1}+ \ldots + c_{k-1}f + c_k = 0 {\rm \ \ \ \ with} \ c_i\in I^i.$$
The set of elements $\overline{I}$ in $S$ which are integral over $I$ is the {\it integral closure}
of $I$. It is known that the integral closure of a monomial ideal $I\subset S$ is a monomial ideal
generated by all monomials $u \in S$ for which there exists an integer $k$ such that
$u^k\in I^k$ (see \cite[Theorem 1.4.2]{hh'}). The ideal $I$ is {\it integrally closed}, if $I = \overline{I}$, and $I$ is {\it normal} if all powers
of $I$ are integrally closed. By \cite[Theorem 3.3.18]{v'}, a monomial ideal $I$ is normal if and only if the Rees algebra $\mathcal{R}(I)=S[IT]=\bigoplus_
{n=0}^{\infty}I^n$ is a normal ring.

Let $I\subset S$ be a monomial ideal. A classical result by Burch \cite{b'} states
that $$\min_k{\rm depth}(S/I^k)\leq n-\ell(I),$$ where $\ell(I)$ is the
analytic spread of $I$, that is, the dimension of $\mathcal{R}(I)/
{{\frak{m}}\mathcal{R}(I)}$, where $\frak{m}=(x_1,\ldots,x_n)$ is the maximal ideal of $S$. By a theorem of
Brodmann \cite{b}, ${\rm depth}(S/I^k)$ is constant for large $k$. We call
this constant value the {\it limit depth} of $I$, and denote it by
$\lim_{k\rightarrow \infty}{\rm depth}(S/I^k)$. Brodmann improved the Burch's
inequality by showing that$$\lim_{k\rightarrow \infty}{\rm depth}(S/I^k)
\leq n-\ell(I).$$It is well-known \cite[Proposition 10.3.2]{hh'} that the equality occurs in the above inequality, if $I$ is a normal ideal. As mention in introduction, recently, Hoa and Trung \cite[Lemma 1.5]{ht} proved that for every monomial ideal $I$, $$\lim_{k\rightarrow \infty}{\rm depth}(S/\overline{I^k)}
= n-\ell(I).$$

Let $I\subseteq S$ be a monomial ideal. The set of minimal monomial generators of $I$ is denoted by $G(I)$ and we set $\mu(I):=|G(I)|$. We also denote the set of associated primes of $S/I$, by ${\rm Ass}(S/I)$. The {\it associated graded ring of $S$ with respect to $I$} will be denoted by ${\rm gr}_I(S)$ and it is defined as ${\rm gr}_I(S)=\bigoplus_{k=0}^{\infty}I^k/I^{k+1}$.

\subsection{Notions from combinatorics}

Let $G$ be a simple graph with vertex set $V(G)=\big\{v_1, \ldots,
v_n\big\}$ and edge set $E(G)$. For every vertex $v_i\subset V(G)$, we denote by $G\setminus v_i$, the graph with vertex set $V(G\setminus v_i)=V(G)\setminus \{v_i\}$ edge set $E(G\setminus v_i)=\{e\in E(G)\mid v_i\notin e\}$. A {\it tree} is a connected graph which has no cycle. The {\it girth} of $G$, denoted by ${\rm girth}(G)$ is the length of the shortest cycle in $G$. We set ${\rm girth}(G)=\infty$, if $G$ has no cycle. The graph $G$ is {\it bipartite} if there exists a partition $V(G)=U_1\cup U_2$ with $U_1\cap U_2=\varnothing$ such that each edge of $G$ is of the form $\{v_i,v_j\}$ with $v_i\in U_1$ and $v_j\in U_2$. A subset $A$ of $V(G)$ is called an {\it independent subset} of $G$ if there are no edges among the vertices of $A$.

A {\it simplicial complex} $\Delta$ on the set of vertices $[n]:=\{1,
\ldots,n\}$ is a collection of subsets of $[n]$ which is closed under
taking subsets; that is, if $F \in \Delta$ and $F'\subseteq F$, then also
$F'\in\Delta$. By $\widetilde{H}_i(\Delta; \mathbb{K})$, we mean the $i$th reduced homology of $\Delta$ with coefficients $\mathbb{K}$.

The {\it independence simplicial complex} of a graph $G$ is defined by
$$\Delta_G=\{A\subseteq V(G)\mid A \,\, \mbox{is an independent set in}\,\,
G\},$$and it is an important object in combinatorial commutative algebra.

\subsection{Notions from combinatorial commutative algebra}

One of the connections between the combinatorics and commutative algebra is
via rings constructed from the combinatorial objects. Let $\Delta$ be a simplicial complex on $[n]$. For every
subset $F\subseteq [n]$, we set $\mathbf{x}_F=\prod_{i\in F}x_i$. The {\it
Stanley--Reisner ideal of $\Delta$ over $\mathbb{K}$} is the ideal $I_{
\Delta}$ of $S$ which is generated by those squarefree monomials $\mathbf{x}_F$ with
$F\notin\Delta$. In other words, $I_{\Delta}=(\mathbf{x}_F\mid F\notin
\Delta)$.

There is a natural correspondence between quadratic squarefree monomial ideals of $S$ and finite simple graphs with $n$ vertices. To every simple graph $G$ with vertex set $V(G)=\{v_1, \ldots, v_n\}$ and edge set $E(G)$, one associates its {\it edge ideal} $I(G)$ defined by
$$I(G)=\big(x_ix_j: \{v_i, v_j\}\in E(G)\big)\subseteq S.$$On can easily check that $I(G)=I_{\Delta_G}$. It is well-known that for any graph $G$, we have $\ell(I(G))=n-p$, where $n$ is the number of vertices and $p$ is the number of bipartite connected components of $G$ (see e.g. \cite[Page 50]{v}).


\section{Stanley depth of integral closure of powers of edge ideals} \label{sec3}

In this section, we study the Stanley depth of integral closure of powers of edge ideals and their quotients. In \cite{s3}, we proved that for every graph $G$ the modules $S/I(G)^k$ and $I(G)^k/I(G)^{k+1}$ satisfy Stanley's inequality for every integer $k\gg 0$. In the same paper, we also proved that for any non-bipartite graph $G$, the ideal $I(G)^k$ satisfies Stanley's inequality for every $k\gg 0$. In this section, we prove all these results are true, if one replaces the powers of $I(G)$ by their integral closure. The following lemma from \cite{s1} has a key role in this section.

\begin{lem} \label{collect}
{\rm (}\cite[Theorem 2.8]{s1}{\rm )} Let $I_2\subseteq I_1$ be two monomial ideals in $S$. Then there  exists an integer $s\geq 1$, such that for every $m\geq 1$,$${\rm sdepth} (I_1^{sm}/I_2^{sm}) \leq {\rm sdepth} (\overline{I_1}/\overline{I_2}).$$
\end{lem}

The following two theorems are the first main results of this section and they follow from Lemma \ref{collect} and the results of \cite{s3}.

\begin{thm} \label{quo}
Let $G$ be a graph with edge ideal $I=I(G)$. Suppose that $p$ is the number of bipartite connected components of $G$. Then for every integer $k\geq 1$, we have ${\rm sdepth}(S/\overline{I^k})\geq p$. In particular, $S/\overline{I^k}$ satisfies Stanley's inequality for every integer $k\gg 0$.
\end{thm}

\begin{proof}
By Lemma \ref{collect}, there  exists an integer $s\geq 1$,
$${\rm sdepth} (S/\overline{I^k})\geq {\rm sdepth} (S/I^{ks}).$$On the other hand, it follows from \cite[Theorem 2.3]{s3} that ${\rm sdepth} (S/I^{ks})\geq p$. This proves the first assertion. The last statement follows from \cite[Lemma 1.5]{ht}, together with the fact that $\ell(I)=n-p$.
\end{proof}

\begin{thm} \label{ide}
Let $G$ be a non-bipartite graph with edge ideal $I=I(G)$. Suppose that $p$ is the number of bipartite connected components of $G$. Then for every integer $k\geq 1$, we have ${\rm sdepth}(\overline{I^k})\geq p+1$. In particular, $\overline{I^k}$ satisfies Stanley's inequality for every integer $k\gg 0$.
\end{thm}

\begin{proof}
The proof is similar to the proof Theorem \ref{quo}. The only difference is that one should use \cite[Corollary 3.2]{s3} instead of \cite[Theorem 2.3]{s3}.
\end{proof}

Assume that $G$ is a bipartite graph. By \cite[Theorem 1.4.6 and Corollay 10.3.17]{hh'}, we know that $I(G)$ is a normal ideal. Thus, the study of the Stanley depth of $\overline{I(G)^k}$ is nothing other than that of $I(G)^k$. We do not know whether for a bipartite graph $G$, the ideal $I(G)^k$ satisfies Stanley's inequality, for any integer $k\gg 0$. However, we proved in \cite[Corollary 3.6]{s3} that $I(G)^k$ satisfies Stanley's conjecture for any integer $k\gg 0$, provided that $G$ has a connected component which is a tree (with at least one edge). In the same paper, we also proposed Question \ref{q4} and proved that $I(G)^k$ satisfies Stanley's inequality, for every bipartite graph $G$ and for every integer $k \gg 0$, provided that the answer of Question \ref{q4} is positive. In \cite[Proposition 3.4]{s3}, we gave a positive answer to this question in the case $G$ is a tree. This result will be generalized in the following theorem.

\begin{thm} \label{girth}
Let $G$ be a connected bipartite graph (with at least one edge) and suppose that ${\rm girth}(G)=g$. Then for every positive integer $k\leq g/2+1$, we have ${\rm sdepth}(I(G)^k)\geq 2$.
\end{thm}

\begin{proof}
If $g=\infty$, i.e., if $G$ is a tree, the assertion follows from \cite[Proposition 3.4]{s3}. Thus, assume that $g$ is finite. As $G$ is a bipartite graph, $g$ is an even integer. Assume that $g=2r$ and let $k$ be a positive integer with $k\leq r+1$. We must prove that ${\rm sdepth}(I(G)^k)\geq 2$. For $k=1$, the desired inequality follows from \cite[Corollary 3.4]{s2}. Thus, assume that $k\geq 2$. We use induction on the number of vertices of $G$, say $n$. Let $C$ be a cycle of $G$ of length $g=2r$. Without lose of generality, we assume that $V(C)=\{v_1, \ldots, v_{2r}\}$ and$$E(C)=\big\{\{v_1, v_2\}, \{v_2, v_3\}, \ldots, \{v_{2r-1}, v_{2r}\}, \{v_1, v_{2r}\}\big\}.$$Let $S_1=\mathbb{K}[x_2, \ldots, x_n]$ be the polynomial ring obtained from $S$ by deleting the variable $x_1$ and consider the ideals $I_1=I(G)^k\cap S_1$ and
$I_1'=(I(G)^k:x_1)$.

Now $I(G)^k=I_1\oplus x_1I_1'$ (as vector spaces) and therefore by definition of  the Stanley depth we have
\[
\begin{array}{rl}
{\rm sdepth}(I(G)^k)\geq \min \{{\rm sdepth}_{S_1}(I_1), {\rm sdepth}_S(I_1')\}.
\end{array} \tag{1} \label{3}
\]

Notice that $I_1=I(G\setminus v_1)^k$. Since$$k\leq \frac{{\rm girth}(G)}{2}+1\leq \frac{{\rm girth}(G\setminus v_1)}{2}+1,$$the induction hypothesis implies that ${\rm sdepth}_{S_1}(I_1)\geq 2$. Thus, using the inequality (\ref{3}), it is enough to prove that ${\rm sdepth}_S(I_1')\geq 2$.

For every integer $i$ with $2\leq i\leq 2k-2$, let $S_i=\mathbb{K}[x_1, \ldots, x_{i-1}, x_{i+1}, \ldots, x_n]$ be the polynomial ring obtained from $S$ by deleting the variable $x_i$ and consider the ideals $I_i'=(I_{i-1}':x_i)$ and $I_i=I_{i-1}'\cap S_i$.

{\bf Claim.} For every integer $i$ with $1\leq i\leq 2k-3$ we have$${\rm sdepth}(I_i')\geq \min \{2, {\rm sdepth}(I_{i+1}')\}.$$

{\it Proof of the Claim.} For every integer $i$ with $1\leq i\leq 2k-3$, we have $I_i'=I_{i+1}\oplus x_{i+1}I_{i+1}'$ (as vector spaces) and therefore by definition of  the Stanley depth we have
\[
\begin{array}{rl}
{\rm sdepth}(I_i')\geq \min \{{\rm sdepth}_{S_{i+1}}(I_{i+1}), {\rm sdepth}_S(I_{i+1}')\},
\end{array} \tag{2} \label{4}
\]

Notice that for every integer $i$ with $1\leq i\leq 2k-3$, we have $I_i'=(I(G)^k:x_1x_2\ldots x_i)$. Thus$$I_{i+1}=I_i'\cap S_{i+1}=((I(G)^k\cap S_{i+1}):_{S_{i+1}}x_1x_2\ldots x_i).$$Hence, using \cite[Proposition 2]{p} (see also \cite[Proposition 2.5]{s4}, we conclude that

\[
\begin{array}{rl}
{\rm sdepth}_{S_{i+1}}(I_{i+1})\geq {\rm sdepth}_{S_{i+1}}(I(G)^k\cap S_{i+1}).
\end{array} \tag{3} \label{5}
\]

Note that $I(G)^k\cap S_{i+1}=I(G\setminus v_i)^k$. Since$$k\leq \frac{{\rm girth}(G)}{2}+1\leq \frac{{\rm girth}(G\setminus v_i)}{2}+1,$$the induction hypothesis implies that ${\rm sdepth}_{S_{i+1}}(I(G)^k\cap S_{i+1})\geq 2$. Hence, the claim follows by inequalities (\ref{4}), and (\ref{5}).

It is clear that $I_{2k-2}'=(I(G)^k:x_1x_2\ldots x_{2k-2})$. Thus, by \cite[Proposition 3.2]{ab}, there exists a bipartite graph $G'$ with $V(G')=V(G)$ such that $I(G')=(I(G)^k:x_1x_2\ldots x_{2k-2})$. Therefore, \cite[Corollary 3.4]{s2} implies that$${\rm sdepth}(I_{2k-2}')={\rm sdepth}((I(G)^k:x_1x_2\ldots x_{2k-2}))={\rm sdepth}(I(G'))\geq 2.$$ Therefore, using the claim repeatedly, we conclude that ${\rm sdepth}(I_1')\geq 2$. This completes the proof of the theorem.
\end{proof}

As a consequence of Theorem \ref{girth}, we obtain the following corollary.

\begin{cor} \label{discon}
Let $G$ be a bipartite graph with at least one edge. Suppose that $p$ is the number of connected components of $G$ and assume that $g$ is the maximum girth of connected components of $G$. Then for every positive integer $k\leq g/2+1$, we have ${\rm sdepth}(I(G)^k)\geq p+1$.
\end{cor}

\begin{proof}
Let $H$ be a connected component of $G$ with ${\rm girth}(H)=g$ and set $S'=\mathbb{K}[x_i\mid v_i\in V(H)]$. It follows from \cite[Theorem 3.1]{s3} and Theorem \ref{girth} that for every positive integer $k\leq g/2+1$,$${\rm sdepth}(I(G)^k)\geq \min_{1\leq l \leq k}\{{\rm sdepth}_{S'}(I(H)^l)\}+p-1\geq p+1.$$
\end{proof}

Let $G$ be an arbitrary graph. Our next goal in this section is to study the Stanley depth of the modules in the form $\overline{I(G)^k}/\overline{I(G)^{k+1}}$. We will see in Corollary \ref{sconj2} that these modules satisfy Stanley's inequality for every integer $k\gg 0$. The proof of this result is also based on Lemma \ref{collect}. However, we first need the following lemma.

\begin{lem} \label{stwoquo}
Let $G$ be a graph with edge ideal $I=I(G)$. Suppose that $p$ is the number of bipartite connected components of $G$. Then for every pair of integers $s>t\geq 0$, we have ${\rm sdepth}(I^t/I^s)\geq p$.
\end{lem}

\begin{proof}
Note that$$I^t/I^s=\bigoplus_{k=t}^{s-1}I^k/I^{k+1}.$$By the definition of Stanley depth we conclude that$${\rm sdepth}(I^t/I^s)\geq \min\big\{{\rm sdepth}(I^k/I^{k+1}) \mid k=t, \ldots, s-1\big\}\geq p,$$where the last inequality follows from \cite[Theorem 2.2]{s3}.
\end{proof}

In the next theorem we will show that the number of bipartite connected components of $G$ is a lower bound for the Stanley depth of $\overline{I(G)^k}/\overline{I(G)^{k+1}}$.

\begin{thm} \label{twoquo}
Let $G$ be a graph with edge ideal $I=I(G)$. Suppose that $p$ is the number of bipartite connected components of $G$. Then for every integer $k\geq 0$, we have ${\rm sdepth}(\overline{I^k}/\overline{I^{k+1}})\geq p$.
\end{thm}

\begin{proof}
By Lemma \ref{collect}, for every integer $k\geq 0$, there exists an integer $s\geq 1$ such that$${\rm sdepth}(\overline{I^k}/\overline{I^{k+1}})\geq {\rm sdepth}(I^{sk}/I^{s(k+1)}),$$(we set $m=1$ in Lemma \ref{collect}). Thus, Lemma \ref{stwoquo} implies that ${\rm sdepth}(\overline{I^k}/\overline{I^{k+1}})\geq p$.
\end{proof}

In Corollary \ref{gtwoquo}, we will prove that for any graph $G$ with $p$ bipartite connected components,$$\lim_{k\rightarrow\infty}{\rm depth}(\overline{I(G)^k}/\overline{I(G)^{k+1}})=p.$$Thus, as a consequence of Theorem \ref{twoquo}, we obtain the following result.

\begin{cor} \label{sconj2}
Let $G$ be a graph with edge ideal $I=I(G)$. Then $\overline{I^k}/\overline{I^{k+1}}$ satisfies Stanley's inequality, for every integer $k\gg 0$.
\end{cor}


\section{Depth of integral closure of powers of monomial ideals} \label{sec4}

In This section, we study the depth of the integral closure of powers of monomial ideals. As we promised in Section \ref{sec3}, our first goal is to prove that for every graph $G$ with $p$ bipartite connected components,$$\lim_{k\rightarrow\infty}{\rm depth}(\overline{I(G)^k}/\overline{I(G)^{k+1}})=p.$$In fact, we prove a more general result in Theorem \ref{dtwoquo}. We show that for any monomial ideal $I\varsubsetneq S$, the sequence $\{{\rm depth}(\overline{I^k}/\overline{I^{k+1}})\}_{k=0}^{\infty}$ is convergent and$$\lim_{k\rightarrow\infty}{\rm depth}(\overline{I^k}/\overline{I^{k+1}})=n-\ell(I).$$

\begin{thm} \label{dtwoquo}
For any nonzero monomial ideal $I\varsubsetneq S$, the sequence $\{{\rm depth}(\overline{I^k}/\overline{I^{k+1}})\}_{k=0}^{\infty}$ is convergent and moreover,$$\lim_{k\rightarrow\infty}{\rm depth}(\overline{I^k}/\overline{I^{k+1}})=n-\ell(I).$$
\end{thm}

\begin{proof}
Note that $A={\rm gr}_I(S)=\bigoplus_{k=0}^{\infty}I^k/I^{k+1}$ is a finitely generated standard graded $S$-algebra. By \cite[Proposition 5.3.4]{hs}, there exists an integer $s\geq 1$ such that for all $k\geq s$ we have $\overline{I^k}=I^{k-s}\overline{I^s}$. This shows that $E=\bigoplus_{k=0}^{\infty}\overline{I^k}/\overline{I^{k+1}}$ is a finitely generated graded $A$-module. Hence, \cite[Theorem 1.1]{hh''} implies that the sequence $\{{\rm depth}(\overline{I^k}/\overline{I^{k+1}})\}_{k=0}^{\infty}$ is convergent.

Let $k_0\geq 1$ be an integer with the property that for every $k\geq k_0$ we have$${\rm depth}(\overline{I^k}/\overline{I^{k+1}})=\lim_{k\rightarrow\infty}{\rm depth}(\overline{I^k}/\overline{I^{k+1}}).$$As mentioned above, there exists an integer $s\geq 1$ such that for all $k\geq s$ we have $\overline{I^k}=I^{k-s}\overline{I^s}$. In particular, for every integer $k\geq 1$, we have$$(\overline{I^s})^k\subseteq\overline{I^{ks}}=I^{(k-1)s}\overline{I^s}=(I^s)^{k-1}\overline{I^s}\subseteq \overline{(I^s)}^{k-1}\overline{I^s}=(\overline{I^s})^k.$$ Hence, $(\overline{I^s})^k=\overline{I^{ks}}$, for every integer $k\geq 1$. Let $k\geq k_0$ be an integer. For every integer $i$ with $ks\leq i\leq (k+1)s-2$, consider the following exact sequence.$$0\longrightarrow \overline{I^{i+1}}/\overline{I^{(k+1)s}}\longrightarrow \overline{I^i}/\overline{I^{(k+1)s}}\longrightarrow \overline{I^i}/\overline{I^{i+1}}\rightarrow 0$$Appliying the depth depth Lemma \cite[Proposition 1.2.9]{bh} on the above exact sequence, we obtain that$${\rm depth}(\overline{I^i}/\overline{I^{(k+1)s}})\geq \min\{{\rm depth}(\overline{I^{i+1}}/\overline{I^{(k+1)s}}), {\rm depth}(\overline{I^i}/\overline{I^{i+1}})\}.$$Using this inequality repeatedly, we conclude that
\begin{align*}
{\rm depth}(\overline{I^{ks}}/\overline{I^{(k+1)s}}) & \geq \min\big\{{\rm depth}(\overline{I^i}/\overline{I^{i+1}}) : i=ks, \ldots, (k+1)s-1\big\}\\ & =\lim_{k\rightarrow\infty}{\rm depth}(\overline{I^k}/\overline{I^{k+1}}),
\end{align*}
where the last equality follows from the choice of $k$. Thus, we have$$\lim_{k\rightarrow\infty}{\rm depth}(\overline{I^s}^k/\overline{I^s}^{k+1})=\lim_{k\rightarrow\infty}{\rm depth}(\overline{I^{ks}}/\overline{I^{(k+1)s}})\geq \lim_{k\rightarrow\infty}{\rm depth}(\overline{I^k}/\overline{I^{k+1}}).$$By \cite[Theorem 1.2]{hh''},$$\lim_{k\rightarrow\infty}{\rm depth}(\overline{I^s}^k/\overline{I^s}^{k+1})=n-\ell(\overline{I^s}),$$and since $\ell(\overline{I^s})=\ell(I^s)=\ell(I),$ we conclude that$$\lim_{k\rightarrow\infty}{\rm depth}(\overline{I^k}/\overline{I^{k+1}})\leq \lim_{k\rightarrow\infty}{\rm depth}(\overline{I^s}^k/\overline{I^s}^{k+1})=n-\ell(I).$$

We now prove that $\lim_{k\rightarrow\infty}{\rm depth}(\overline{I^k}/\overline{I^{k+1}})\geq n-\ell(I)$.

Consider the exact sequence$$0\longrightarrow \overline{I^k}/\overline{I^{k+1}}\longrightarrow S/\overline{I^{k+1}}\longrightarrow S/\overline{I^k}\rightarrow 0.$$By \cite[Lemma 1.5]{ht}, there exists an integer $k_1\geq 1$ such that for every $k\geq k_1$, we have ${\rm depth}(S/\overline{I^k})=n-\ell(I)$. Thus, applying the depth Lemma \cite[Proposition 1.2.9]{bh} on the above exact sequence, we conclude that ${\rm depth}(\overline{I^k}/\overline{I^{k+1}})\geq n-\ell(I)$, for every integer $k\ge k_1$. This completes the proof.
\end{proof}

Restricting to edge ideals, we obtain the following corollary.

\begin{cor} \label{gtwoquo}
For any graph $G$, the sequence $\{{\rm depth}(\overline{I(G)^k}/\overline{I(G)^{k+1}})\}_{k=0}^{\infty}$ is convergent and$$\lim_{k\rightarrow\infty}{\rm depth}(\overline{I(G)^k}/\overline{I(G)^{k+1}})=p,$$where $p$ is the number of bipartite connected components of $G$.
\end{cor}

As we saw in Section \ref{sec3}, Lemma \ref{collect} has a key role in the proof of Theorems \ref{quo}, \ref{ide} and \ref{twoquo}. As a particular case of this lemma, for every monomial ideal $I$, there exists an integer $s\geq 1$ with the property that$${\rm sdepth}(S/I^{sm})\leq {\rm sdepth}(S/\overline{I}).$$It is reasonable to ask whether this inequality is true, if one replaces sdepth by depth. In Theorem \ref{dnormal1}, we show this is the case. Our proof is base on a formula due to Takayama \cite{t1}, which is presented as follows.

Let $I$ be a monomial ideal. As $S/I$ is a $\mathbb{Z}^n$-graded $S$-module, it follows that for every integer $i$, the local cohomology module $H_{\mathfrak{m}}^i(S/I)$ is $\mathbb{Z}^n$-graded too. For any vector $\alpha\in \mathbb{Z}^n$, we denote the $\alpha$-component of $H_{\mathfrak{m}}^i(S/I)$ by $H_{\mathfrak{m}}^i(S/I)_{\alpha}$. The {\it co-support} of the vector $\alpha=(\alpha_1, \ldots, \alpha_n)$ is defined to be the set ${\rm CS}(\alpha)=\{i : \alpha_i < 0\}$. For any subset $F\subseteq [n]$, let $S_F=S[x_i^{-1}: i\in F]$. Suppose that $\Delta(I)$ is the simplicial complex over $[n]$ with Stanley-Reisner ideal $I_{\Delta(I)}=\sqrt{I}$. For any vector $\alpha=(\alpha_1, \ldots, \alpha_n)\in \mathbb{Z}^n$, we set$$\Delta_{\alpha}(I)=\{F\subseteq [n]\setminus {\rm CS}(\alpha) : \mathbf{x}^{\alpha}\notin IS_{F\cup {\rm CS}(\alpha)}\},$$where $\mathbf{x}^{\alpha}=x_1^{\alpha_1}\ldots x_n^{\alpha_n}$. Takayama \cite[Theorem 2.2]{t1} proves that for every vector $\alpha\in \mathbb{Z}^n$ and for every integer $i$, we have
\[
\begin{array}{rl}
\dim_{\mathbb{K}} H_{\mathfrak{m}}^i(S/I)_{\alpha}=\dim_{\mathbb{K}} \widetilde{H}_{i-\mid{\rm CS}(\alpha)\mid-1}(\Delta_{\alpha}(I); \mathbb{K}).
\end{array} \tag{4} \label{2}
\]
Using this formula, we are able to prove the following proposition.

\begin{prop} \label{dnormal}
Let $I$ be a monomial ideal in $S$. Assume that $s\geq 1$ is an integer with property the for every monomial $u\in \overline{I}$, we have $u^s\in I^s$. Then for every $m\geq 1$,
$${\rm depth} (S/I^{sm}) \leq {\rm depth} (S/\overline{I}).$$
\end{prop}

\begin{proof}
Set $t={\rm depth} (S/\overline{I})$. It follows that there exists a vector $\alpha\in \mathbb{Z}^n$ such that $H_{\mathfrak{m}}^t(S/\overline{I})_{\alpha}\neq 0$. Thus, equality (\ref{2}), implies that$$\widetilde{H}_{t-\mid{\rm CS}(\alpha)\mid-1}(\Delta_{\alpha}(\overline{I}); \mathbb{K})\neq 0.$$ Now it follows from the assumption that for every integer $m\geq 1$ and every monomial $u\in S$, we have $u\in \overline{I}$ if and only if $u^{sm}\in I^{sm}$. We conclude that
\begin{align*}
& \Delta_{\alpha}(\overline{I})=\{F\subseteq [n]\setminus {\rm CS}(\alpha) : \mathbf{x}^{\alpha}\notin \overline{I}S_{F\cup {\rm CS}(\alpha)}\}\\ & =\{F\subseteq [n]\setminus {\rm CS}(\alpha) : \mathbf{x}^{sm\alpha}\notin I^{sm}S_{F\cup {\rm CS}(\alpha)}\}\\ & =\Delta_{sm\alpha}(I^{sm}),
\end{align*}
where the last equality follows from the fact that ${\rm CS}(sm\alpha)={\rm CS}(\alpha)$. Therefore,$$\widetilde{H}_{t-\mid{\rm CS}(sm\alpha)\mid-1}(\Delta_{sm\alpha}(I^{sm}); \mathbb{K})=\widetilde{H}_{t-\mid{\rm CS}(\alpha)\mid-1}(\Delta_{\alpha}(\overline{I}); \mathbb{K})\neq 0.$$By equality (\ref{2}), we deduce that $H_{\mathfrak{m}}^t(S/I^{sm})_{sm\alpha}\neq 0$ and hence, ${\rm depth} (S/I^{sm}) \leq t$.
\end{proof}

Our next goal is determine an integer $s$ which satisfies the assumption of Proposition \ref{dnormal}. The following lemma is the main step in this regard. The proof of this lemma is based on the determinantal trick which was suggested to us by Irena Swanson.

\begin{lem} \label{powe}
Let $I$ be a monomial ideal with analytic spread $\ell=\ell(I)$. Assume that $u\in \overline{I}$ is a monomial. Then there exists an integer $t\leq \mu(\overline{I^{\ell-1}})$ with $u^t\in I^t$.
\end{lem}

\begin{proof}
Assume that $G(\overline{I^{\ell-1}})=\{u_1, \ldots, u_m\}$ is the set of minimal monomial generators of $\overline{I^{\ell-1}}$. By \cite[Theorem 5.1]{s'}, there exists a monomial ideal $J\subseteq I$, such that $\overline{I^{\ell}}=J\overline{I^{\ell-1}}$. In particular, $\overline{I^{\ell}}\subseteq I\overline{I^{\ell-1}}$. Since $u\in \overline{I}$, we conclude that $u\overline{I^{\ell-1}}\subseteq \overline{I^{\ell}}\subseteq I\overline{I^{\ell-1}}$. As a consequence, for every integer $i$ with $1\leq i \leq m$, we have $uu_i\in I\overline{I^{\ell-1}}$. Thus, we may write $uu_i=\sum_{j=1}^mc_{ij}u_j$, for some polynomial $c_{ij}\in I$. Let $A=(a_{ij})_{m\times m}$ be the matrix with $a_{ij}=\delta_{ij}u-c_{ij}$, where $\delta_{ij}$ denotes the Kronecker delta function. Set $\mathbf{u}=(u_1, \ldots, u_m)^T$. Then we have $A\mathbf{u}=0$. Multiplying by the adjoint of $A$, we conclude that $\det(A)u_i=0$, for every integer $i$ with $1\leq i\leq m$. This means that $\det(A)=0$.

Obviously, $\det(A)$ is a polynomial of the form$$\det(A)=u^m+f_1u^{m-1}+ \cdots + f_{m-1}u+f_m,$$where for every integer $i$ with $1\leq i\leq m$, we have $f_i\in I^i$. In particular, we obtain that
\[
\begin{array}{rl}
u^m+f_1u^{m-1}+ \cdots + f_{m-1}u+f_m=0.
\end{array} \tag{5} \label{pow}
\]
For every integer $i$ with $1\leq i\leq m$, let $\alpha_i \in \mathbb{K}$ be the coefficient of $u^i$ in the polynomial $f_i$. Thus, equality (\ref{pow}) implies that$$u^m+\alpha_1u^m+ \cdots + \alpha_{m-1}u^m+\alpha_mu^m=0.$$Hence, that there exists and integer $t$ with $1\leq t\leq m$ such that $\alpha_t\neq 0$. This means that $u^t$ is one of the monomials appearing in the expansion of $f_t$. Since $f_t\in I^t$ and $I^t$ is a monomial ideal, we conclude that $u^t\in I^t$. We also notice that $t\leq m=\mu(\overline{I^{\ell-1}})$, which completes the proof.
\end{proof}

As a consequence of Proposition \ref{dnormal} and Lemma \ref{powe}, we obtain our next main result.

\begin{thm} \label{dnormal1}
Let $I\subset S$ be a monomial ideal with analytic spread $\ell=\ell(I)$. Set $s=\mu(\overline{I^{\ell-1}})!$. Then for every integer $m\geq 1$, we have
$${\rm depth} (S/I^{sm}) \leq {\rm depth} (S/\overline{I}).$$
\end{thm}

\begin{proof}
Let $G(\overline{I})=\{u_1, \ldots, u_r\}$ be the set of minimal monomial generators of $\overline{I}$. By Lemma \ref{powe}, for every $1\leq i \leq r$, there exists integer $k_i\leq \mu(\overline{I^{\ell-1}})$, such that $u_i^{k_i}\in I^{k_i}$. As $k_i$ divides $s$ for every $i$, we conclude that $u^s\in I^s$, for every monomial $u\in \overline{I}$. The assertion now follows from Proposition \ref{dnormal}.
\end{proof}

Assume that $I$ is a squarefree monomial ideal. We know from the proof of \cite[Theorem 2.6]{htt} that$${\rm depth}(S/I^m)\leq {\rm depth}(S/I),$$for every integer $m\geq 1$. In the following theorem, we show that the above inequality is true for any integrally closed monomial ideal.

\begin{thm} \label{dnormal2}
Let $I$ be an integrally closed monomial ideal in $S$. Then for every integer $m\geq 1$, we have
$${\rm depth} (S/I^m) \leq {\rm depth} (S/I).$$
\end{thm}

\begin{proof}
As $\overline{I}=I$, for every $u\in \overline{I}$ we have $u\in I$ and hence, Proposition \ref{dnormal} implies the assertion.
\end{proof}

Let $P = (x_{i_1}, \ldots, x_{i_r})$ be a monomial prime ideal in $S$, and $I\subseteq S$ any monomial ideal. Set $L=[n]\setminus\{x_{i_1}, \ldots, x_{i_r}\}$. We denote by $I(P)$ the monomial ideal in the polynomial ring $S(P) = \mathbb{K}[x_{i_1}, \ldots, x_{i_r}]$, which is obtained from $I$ by applying the $\mathbb{K}$-algebra homomorphism $S\rightarrow S(P)$ defined by $x_i\mapsto 1$ for all $i\in L$ and $x_i\mapsto x_i$, otherwise. It is known that (\cite[Lemma 1.3]{hrv})
$${\rm Ass}(S(P)/I(P))=\{Q\in {\rm Ass}(S/I): x_i\notin Q {\rm \ for \ all \ } i\in L\}.$$
Using this fact, we deduce the following result concerning the associated primes of powers of integrally closed monomial ideals.

\begin{cor} \label{ass}
Let $I$ be an integrally closed monomial ideal in $S$. Then for every integer $m\geq 1$, we have
$${\rm Ass}(S/I)\subseteq {\rm Ass}(S/I^m).$$
\end{cor}

\begin{proof}
Let $P\in {\rm Ass}(S/I)$ be a monomial prime ideal of $S$. Then by \cite[Lemma 1.3]{hrv}, we have $P\in {\rm Ass}(S(P)/I(P))$ and hence, ${\rm depth}_{S(P)}(S(P)/I(P))=0$. It follows from \cite[Lemma 3.1]{s1} that $I(P)$ is an integrally closed ideal and thus, Theorem \ref{dnormal2} shows that ${\rm depth}_{S(P)}(S(P)/I(P)^m)=0$, for every integer $m\geq 1$. Therefore,$$P\in {\rm Ass}(S(P)/I(P)^m)={\rm Ass}(S(P)/I^m(P)).$$Again \cite[Lemma 1.3]{hrv} implies that $P\in {\rm Ass}(S/I^m)$.
\end{proof}


\section*{Acknowledgment}

The author thanks Irena Swanson for suggesting determinantal trick in the proof of Lemma \ref{powe}.




\begin{thebibliography}{10}

\bibitem{ab} A. Alilooee, A. Banerjee, Powers of edge ideals of regularity three bipartite graphs, {\it J. Commut. Algebra}, {\bf 9} (2017), 441--454.

\bibitem {asy} A. Alipour, S. A. Seyed Fakhari, S. Yassemi, Stanley depth
 of factor of  polymatroidal ideals and edge ideal of  forests, {\it Arch. Math. (Basel)}, {\bf 105} (2015), 323--332.

\bibitem {b} M. Brodmann, The asymptotic nature of the analytic spread,
    {\it Math. Proc. Cambridge Philos. Soc.} {\bf 86} (1979), no. 1,
    35--39.

\bibitem {bh} W. Bruns, J. Herzog, {\it Cohen--Macaulay Rings}, Cambridge
    Studies in Advanced Mathematics, {\bf 39}, Cambridge University Press,
    1993.

\bibitem {b'} L. Burch, Codimension and analytic spread, {\it Proc.
    Cambridge Philos. Soc.} {\bf 72} (1972), 369--373.

\bibitem {abcj} A. M. Duval, B. Goeckner, C. J. Klivans, J. L. Martin, A non-partitionable Cohen-Macaulay simplicial complex, {\it Adv. Math.} {\bf 299} (2016), 381--395.

\bibitem {h} J. Herzog, A survey on Stanley depth. In "Monomial Ideals, Computations and Applications", A. Bigatti,
    P. Gim${\rm\acute{e}}$nez, E. S${\rm\acute{a}}$enz-de-Cabez${\rm\acute{o}}$n (Eds.), Proceedings of MONICA 2011. Lecture Notes in
    Math. {\bf 2083}, Springer (2013).

\bibitem {hh'} J. Herzog, T. Hibi, {\it Monomial Ideals}, Springer-Verlag,
    2011.

\bibitem {hh''} J. Herzog, T. Hibi, The depth of powers of an
ideal, {\it J. Algebra} {\bf 291} (2005), no. 2, 325--650.

\bibitem {hrv} J. Herzog, A. Rauf, M. Vladoiu, The stable set of associated
    prime ideals of a polymatroidal ideal, {\it J. Algebraic Combin.}, {\bf 37} (2013), 289--312.

\bibitem {htt} J. Herzog, Y. Takayama, N. Terai, On the radical of a
    monomial ideal, {\it Arch. Math. (Basel)} {\bf 85} (2005), no. 5,
    397--408.

\bibitem {ht} L. T. Hoa, T. N. Trung, Stability of depth and Cohen-Macaulayness of integral closures of powers of monomial ideals, {\it Acta Math. Vietnam.} {\bf 43} (2018), 67--81.

\bibitem {hs} C. Huneke and I. Swanson, {\it Integral Closure of Ideals Rings, and Modules}, London Math. Soc., Lecture Note
Series {\bf 336}, Cambridge University Press, Cambridge, 2006.

\bibitem {p} D. Popescu, Bounds of Stanley depth, {\it An. St. Univ. Ovidius. Constanta}, 19(2),(2011), 187--194.

\bibitem {psty} M. R. Pournaki, S. A. Seyed Fakhari, M. Tousi, S. Yassemi,
    What is $\ldots$ Stanley depth? {\it Notices Amer. Math. Soc.} {\bf 56}
    (2009), no. 9, 1106--1108.

\bibitem {psy} M. R. Pournaki, S. A. Seyed Fakhari,  S. Yassemi,
    Stanley depth of powers of the edge ideal of a forest,  {\it Proc. Amer. Math. Soc.} {\bf 141}
    (2013), no. 10, 3327--3336.

\bibitem {s1} S. A. Seyed Fakhari, Stanley depth of the integral closure of
    monomial ideals, {\it Collect. Math.} {\bf 64} (2013), 351--362.

\bibitem{s2} S. A. Seyed Fakhari, Stanley depth of weakly polymatroidal ideals and squarefree monomial ideals, {\it  Illinois J. Math.}, {\bf 57} (2013), no. 3, 871--881.

\bibitem {s6} S. A. Seyed Fakhari, Depth, Stanley depth and regularity of ideals associated to graphs, {\it Arch. Math. (Basel)} {\bf 107} (2016), 461--471.

\bibitem{s3} S. A. Seyed Fakhari, On the Stanley depth of powers of edge ideals, {\it J. Algebra}, {\bf 489} (2017), 463--474.

\bibitem{s5} S. A. Seyed Fakhari, Depth and Stanley depth of symbolic powers of cover ideals of graphs, {\it J. Algebra}, {\bf 492} (2017), 402--413.

\bibitem {s4} S. A. Seyed Fakhari, Stanley depth and symbolic powers of monomial ideals, {\it Math. Scand.} {\bf 120} (2017), 5--16.

\bibitem {s'}  P. Singla, Minimal monomial reductions and the reduced fiber ring of an extremal ideal, {\it  Illinois J. Math.}, {\bf 51} (2007), no. 4, 1085--1102.

\bibitem {s} R. P. Stanley, Linear Diophantine equations and local
    cohomology, {\it Invent. Math.} {\bf 68} (1982), no. 2, 175--193.

\bibitem {t1} Y. Takayama, Combinatorial characterizations of generalized Cohen-Macaulay monomial ideals,
{\it Bull. Math. Soc. Sci. Math. Roumanie} (N.S.) {\bf 48} (2005), 327--344.

\bibitem {v} W. Vasconcelos, {\it Integral Closure, Rees Algebras,
    Multiplicities, Algorithms}, Springer Monographs in Mathematics,
    Springer-Verlag, Berlin, 2005.

\bibitem {v'} R. H. Villarreal, {\it Monomial Algebras}, Dekker, New York, N.Y., 2001.

\end{thebibliography}
\end{document}